\documentclass[letterpaper,11pt]{article}

\usepackage[T1]{fontenc}
\usepackage{lmodern}
\usepackage[letterpaper,left=1in,right=1in,top=1in,bottom=1in]{geometry}
\usepackage{enumerate}
\usepackage{amsthm,amsmath,amsfonts,amssymb}
\usepackage{algorithm}
\usepackage{algorithmic}
\usepackage{mathtools}
\usepackage{needspace}
\usepackage{cite}
\usepackage{setspace}
\usepackage{titlesec}
\usepackage{etoolbox}
\usepackage{xurl}
\usepackage[nopatch=footnote]{microtype}
\usepackage[
  colorlinks=true,
  linkcolor=blue,
  citecolor=blue,
  urlcolor=blue,
  anchorcolor=blue,
  breaklinks=true
]{hyperref}
\allowdisplaybreaks[3]
\titleformat{\section}{\large\bfseries}{\thesection.}{0.6em}{}
\titlespacing*{\section}{0pt}{2.1ex plus 0.5ex minus 0.2ex}{0.9ex plus 0.2ex}
\titleformat{\subsection}{\normalsize\bfseries}{\thesubsection.}{0.6em}{}
\titlespacing*{\subsection}{0pt}{1.7ex plus 0.4ex minus 0.2ex}{0.7ex plus 0.2ex}
\titleformat{\paragraph}[runin]{\normalsize\bfseries}{}{0pt}{}
\titlespacing*{\paragraph}{0pt}{1.3ex plus 0.3ex minus 0.2ex}{0.7em}
\apptocmd{\appendix}{%
  \titleformat{\section}{\large\bfseries}{Appendix~\thesection.}{0.6em}{}%
}{}{}

\makeatletter
\renewcommand{\@maketitle}{%
  \newpage\null\vskip 0.5em
  \begin{center}%
    {\Large\bfseries \@title\par}%
    \ifx\@author\@empty\else
      \vskip 1.0em
      {\normalsize\lineskip .5em
       \begin{tabular}[t]{c}\@author\end{tabular}\par}%
    \fi
    \ifx\@date\@empty\else
      \vskip 0.65em{\small \@date\par}%
    \fi
  \end{center}%
  \par\vskip 0.65em
}
\makeatother
\renewenvironment{abstract}{%
  \par\begingroup\small
  \noindent\textbf{Abstract.}\enspace\ignorespaces
}{\par\endgroup\vspace{0.55\baselineskip}}
\newenvironment{keywords}{%
  \par\begingroup\small
  \noindent\textbf{Key words.}\enspace\ignorespaces
}{\par\endgroup\vspace{0.3\baselineskip}}
\newenvironment{msc}{%
  \par\begingroup\small
  \noindent\textbf{MSC codes.}\enspace\ignorespaces
}{\par\endgroup\vspace{0.4\baselineskip}}

\makeatletter
\renewenvironment{thebibliography}[1]{%
  \par\addvspace{1.5\baselineskip}%
  \begin{center}\normalfont\small\bfseries REFERENCES\end{center}%
  \par\nobreak\vspace{0.2\baselineskip}%
  \fontsize{9}{10.5}\selectfont
  \list{\@biblabel{\@arabic\c@enumiv}}{%
    \settowidth\labelwidth{\@biblabel{#1}}%
    \setlength{\labelsep}{0.6em}%
    \setlength{\leftmargin}{\labelwidth}%
    \addtolength{\leftmargin}{\labelsep}%
    \setlength{\itemindent}{0pt}%
    \setlength{\itemsep}{0.8pt plus 0.2pt minus 0.2pt}%
    \setlength{\parsep}{0pt}%
    \setlength{\topsep}{0pt}%
    \setlength{\partopsep}{0pt}%
    \usecounter{enumiv}%
    \let\p@enumiv\@empty
    \renewcommand\theenumiv{\@arabic\c@enumiv}}%
  \sloppy\clubpenalty=10000\widowpenalty=10000
  \sfcode`\.=1000\relax
}{%
  \def\@noitemerr{\@latex@warning{Empty `thebibliography' environment}}%
  \endlist
}
\makeatother

\theoremstyle{plain}
\newtheorem{theorem}{Theorem}[section]
\newtheorem{lemma}[theorem]{Lemma}
\newtheorem{proposition}[theorem]{Proposition}
\newtheorem{corollary}[theorem]{Corollary}

\theoremstyle{definition}

\theoremstyle{remark}
\newtheorem{remark}[theorem]{Remark}

\newcommand{\RR}{\mathbb{R}}

\newcommand{\EE}{\mathbb{E}}

\newcommand{\cv}{\mathrm{conv}}
\newcommand{\dist}{\mathrm{dist}}

\newcommand{\df}{\mathrm{d}}

\newcommand{\PP}{\mathbb{P}}
\newcommand{\BB}{\mathbb{B}}

\newcommand{\bx}{\mathbf{x}}

\newcommand{\by}{\mathbf{y}}
\newcommand{\bw}{\mathbf{w}}

\newcommand{\bg}{\mathbf{g}}
\newcommand{\bz}{\mathbf{z}}
\newcommand{\bv}{\mathbf{v}}
\newcommand{\bu}{\mathbf{u}}

\newcommand{\bq}{\mathbf{q}}
\newcommand{\ba}{\mathbf{a}}

\newcommand{\norm}[1]{\|#1\|}

\newcommand{\interior}{\mathrm{int}}

\newcommand{\sphere}{\mathbb{S}}

\DeclareMathOperator{\Lip}{Lip}

\newcommand{\ip}[2]{\left\langle #1,#2\right\rangle}

\numberwithin{equation}{section}

\hypersetup{
  pdftitle={Exponential Deterministic Query Complexity of Goldstein Stationarity},
  pdfauthor={Yixi Ding and Ya-xiang Yuan},
  pdfsubject={Deterministic local-oracle complexity and accuracy-dependent Goldstein stationarity bounds},
  pdfkeywords={Goldstein stationarity, deterministic oracle complexity, nonsmooth nonconvex optimization, local oracles},
  bookmarksnumbered=true
}

\title{Exponential Deterministic Query Complexity\\of Goldstein Stationarity}
\author{Yixi Ding\thanks{LSEC, ICMSEC, Academy of Mathematics and Systems Science, Chinese Academy of Sciences, Beijing 100190, China; University of Chinese Academy of Sciences, Beijing 100049, China. Email: \texttt{dingyixi@amss.ac.cn}.}\qquad Ya-xiang Yuan\thanks{LSEC, ICMSEC, Academy of Mathematics and Systems Science, Chinese Academy of Sciences, Beijing 100190, China. Email: \texttt{yyx@lsec.cc.ac.cn}.}}
\date{}

\begin{document}

\setcounter{footnote}{1}
\maketitle

\begin{abstract}
We study the deterministic query complexity of finding approximate Goldstein stationary points of globally Lipschitz functions that may be nonsmooth and nonconvex. Under prescribed bounds on the Lipschitz constant and on the difference between the initial function value and the infimum, we prove a lower bound on the worst-case number of oracle calls that is exponential in the dimension at fixed sufficiently small accuracy parameters. The result holds for a local oracle and gives explicit dependence on the accuracy parameters. A complementary deterministic algorithm using only function values gives an exponential upper bound, establishing the exponential order in the dimension in this regime. For a coarser stationarity requirement, we also give a deterministic first-order algorithm with a dimension-free query bound. Thus, the dimension dependence differs between the two accuracy regimes.
\end{abstract}

\begin{keywords}
Goldstein stationarity, deterministic oracle complexity, nonsmooth nonconvex optimization, local oracles
\end{keywords}
\begin{msc}
90C60, 90C56, 68Q25
\end{msc}

\section{Introduction}
\label{sec:introduction}

Nonsmooth nonconvex optimization problems arise, for example, in piecewise affine models and neural networks with nonsmooth activation functions \cite{DNN,nogas}. For these problems, approximate stationarity provides a target that is weaker than global optimality. Unlike in smooth optimization, where a small gradient norm gives a natural measure of stationarity, the nonsmooth setting requires care in choosing a stationarity criterion for which finite-time guarantees are possible.

We consider a globally Lipschitz objective $f:\RR^d\to\RR$ that is bounded below, without imposing convexity or smoothness. For objectives with a Lipschitz continuous gradient, gradient descent admits finite bounds on the number of gradient evaluations needed to obtain a small gradient norm. These bounds are independent of $d$ when the gradient Lipschitz constant, initial function-value gap, and required accuracy are fixed \cite{nogas,DNN}. For a merely Lipschitz objective, however, asking for a small Clarke subgradient at the reported point can preclude a uniform finite-query guarantee over the function class \cite{CFS,KS33}. Thus, simply replacing the gradient by a generalized derivative does not resolve the computational difficulty. Goldstein's stationarity condition relaxes this pointwise requirement by allowing a convex combination of subgradients at nearby points \cite{Goldstein}. More precisely, a point $\bx$ is $(\delta,\epsilon)$-Goldstein stationary if a convex combination of Clarke subgradients at points within distance $\delta>0$ of $\bx$ has norm at most $\epsilon>0$. Thus, $\delta$ determines the neighborhood from which subgradients are taken, whereas $\epsilon$ sets the allowable norm of such a convex combination. Section~\ref{sec:model} gives the formal definitions.

This condition has a direct connection to descent. Goldstein's method uses a minimum-norm element of the convex hull of nearby subgradients: either its norm is small, or its negative direction yields a decrease in the objective \cite{Goldstein}. The difficulty is to obtain enough information to choose such a direction from finitely many oracle calls. A first-order oracle at one point does not supply the convex hull of subgradients throughout a neighborhood. Consequently, a bound on the number of these conceptual descent steps does not by itself give a bound on oracle queries. Randomized first-order algorithms address this difficulty. For any fixed failure probability $\gamma\in(0,1)$, they can find a $(\delta,\epsilon)$-Goldstein stationary point with probability at least $1-\gamma$ using $\widetilde O(\delta^{-1}\epsilon^{-3})$ oracle calls, under fixed Lipschitz and initial-gap bounds; see Davis et al.\ \cite[Theorem~6]{Davis22} and Tian, Zhou, and So \cite[Theorem~3.4]{TS22}. Here $\widetilde O$ suppresses logarithmic factors. Thus, for any fixed success probability below one, the query complexity is polynomially bounded in $d$, $\delta^{-1}$, and $\epsilon^{-1}$, and is in fact independent of $d$. These guarantees use exact first-order information and do not require a Lipschitz continuous gradient.

For deterministic first-order algorithms, Kornowski and Shamir \cite{KS22} and Jordan et al.\ \cite{DNN} prove lower bounds on the worst-case number of oracle calls that grow linearly with $d$ at sufficiently small fixed accuracy parameters. Tian and So \cite{nogas} obtain a dimension-dependent lower bound even with access to a more informative local oracle. These results establish a limitation of deterministic algorithms, rather than only of a particular way to remove random sampling from an existing method. They do not, however, determine whether polynomial dependence on $d$ is possible. Jordan et al.\ \cite{DNN} asked whether a deterministic first-order algorithm can find a $(\delta,\epsilon)$-Goldstein stationary point with a number of oracle calls polynomial jointly in $d$, $\delta^{-1}$, and $\epsilon^{-1}$. At fixed accuracy parameters, the distinction between polynomial and exponential dimension dependence remains important even after a dimension-free guarantee has been ruled out. Our results address this question for the general Lipschitz function class.

\subsection{Contributions}
\label{sec:intro-results}

Let $L>0$ bound the Lipschitz constant, and let $\Delta>0$ bound the initial function-value gap $f(\bx_0)-\inf f$ at a prescribed initial point $\bx_0$. Our results describe how deterministic query complexity depends on the dimension and the accuracy parameters. For $0<\delta<\Delta/L$ and $0<\epsilon<L/3$, Theorem~\ref{thm:main} gives an explicit query complexity lower bound that is exponential in $d$ when $L,\Delta,\delta,\epsilon$ are fixed. It applies to every deterministic algorithm using the local oracle defined in Tian and So \cite{nogas}. Query locations are unrestricted, and the output need not have been queried. In particular, the result rules out a guarantee polynomial jointly in $d$, $\delta^{-1}$, and $\epsilon^{-1}$. The proof uses a resisting-oracle construction that accommodates exponentially many query points; Sections~\ref{sec:construction}--\ref{sec:replay} develop the construction and the argument.

A complementary deterministic algorithm uses only function values and gives an upper bound on the number of calls sufficient to find a $(\delta,\epsilon)$-Goldstein stationary point for every admissible objective (Theorem~\ref{thm:upper}). Together, the bounds give $\exp(\Theta(d))$ query complexity for each fixed choice of $L,\Delta,\delta,\epsilon$ in the range above; see Corollary~\ref{cor:theta-d} and the discussion following it. This determines the exponential order in $d$, not the exact exponential base or the optimal joint dependence on dimension and accuracy.

The conclusion changes when the stationarity condition permits a larger value of $\epsilon$. For any $\delta>0$ and $L/\sqrt2<\epsilon<L$, Proposition~\ref{prop:coarse-upper} gives a deterministic first-order algorithm with a query bound independent of $d$, without imposing smoothness. Here the first-order oracle returns a function value and one Clarke subgradient at each query. We refer to this range of $\epsilon$ as coarse accuracy. Our bounds do not determine the dimension dependence for the intermediate range $1/3\le\epsilon/L\le1/\sqrt2$.

\subsection{Related work}
\label{sec:related-work}

\paragraph{Randomized algorithms and stochastic oracles.}
Burke, Lewis, and Overton \cite{BLO2005} used convex combinations of sampled gradients and proved convergence under their regularity assumptions; Kiwiel \cite{Kiwiel2007} strengthened the analysis and studied a modified algorithm. Zhang et al.\ \cite{CFS} obtained a dimension-free finite-time bound under Hadamard directional differentiability and a direction-dependent first-order oracle. Davis et al.\ \cite{Davis22} and Tian, Zhou, and So \cite{TS22} developed randomized procedures for general Lipschitz functions using exact values and gradients at differentiability points. For zeroth-order information, Lin, Zheng, and Jordan \cite{LZJ2022} related uniform smoothing to the Goldstein subdifferential and developed randomized methods based on function values.

Random choices made by an algorithm are distinct from randomness in a stochastic oracle, which supplies sampled or noisy objective information. Cutkosky, Mehta, and Orabona \cite{O2NC} improved first-order query bounds through an online-to-non-convex conversion using unbiased stochastic gradients with bounded variance. Their neighborhood-stationarity condition requires the nearby points to average to the reported point and implies Goldstein stationarity for Lipschitz objectives. Kornowski and Shamir \cite{KS24} obtained a zeroth-order query bound linear in $d$ for stochastic objectives accessed through sampled function values. The zeroth-order upper bound established below instead uses a deterministic algorithm and exact function values.

\paragraph{Deterministic algorithms and local information.}
Tian and So \cite{nogas} also prove a separate impossibility result for general zero-respecting algorithms, which do not explore coordinates along which the objective is locally constant near all preceding query points. Jordan et al.\ \cite{DNN} show that removing function values precludes a uniform finite deterministic guarantee. Deterministic algorithms for smooth functions with logarithmic dependence on the gradient Lipschitz constant are given in \cite{KS22,DNN}. The present lower bound assumes no such smoothness parameter and places no restriction on query locations.

\paragraph{Concurrent work.}
After we finished our work, we noticed there is a recent paper by Kornowski \cite{Kornowski2026}, which also establishes a lower bound on deterministic query complexity that is exponential in the dimension for Goldstein stationarity under a local oracle.

\paragraph{Other local targets and verification.}
Kornowski and Shamir \cite{KS33} prove a query complexity lower bound, exponential in the dimension, for approaching a point with a small Clarke subgradient, even with randomization and local information. That target is stronger than Goldstein stationarity: a convex combination of nearby subgradients can have small norm even when none of its constituent subgradients does. Kornowski, Padmanabhan, and Shamir \cite{HMLG} establish exponential difficulties for local guarantees formulated in terms of function values. Kalavasis, Stavropoulos, and Zampetakis \cite{Verification2024} study verification complexity and derive search consequences, including an exponential lower bound on the number of zeroth-order queries for \emph{uniform Goldstein stationarity}. Their result for ordinary Goldstein stationarity concerns the resolution of lattice-restricted queries. These results concern different combinations of stationarity target, oracle information, and query restrictions.

\paragraph{DC functions and quantitative assumptions.}
Kong and Lewis \cite{KongLewis} give a deterministic method for a class that includes differences of convex (DC) functions. Their guarantee depends on a directional oracle and an objective-specific measure of nonconvexity. The functions constructed in our lower-bound proof are also DC, so the comparison requires attention to the quantitative assumptions and the oracle information, not only to DC membership. Section~\ref{sec:scope} and Appendix~\ref{app:dc} give this comparison.

\Needspace{7\baselineskip}
\paragraph{Organization.}
Section~\ref{sec:model} defines the problem and oracle model and states the complexity bounds. Section~\ref{sec:construction} gives an overview of the construction and proves its geometric and interpolation properties. Section~\ref{sec:replay} then chooses the scalar parameters and proves the query complexity lower bound. Section~\ref{sec:upper} proves the upper bounds, and Section~\ref{sec:scope} discusses their scope and remaining questions. The appendices cover initialization and stopping, finite direction-net construction, DC structure, and a non-weakly-convex example.

\section{Problem formulation and complexity bounds}
\label{sec:model}

We recall the stationarity criterion, specify the function class and oracle model, and state the query complexity bounds.

\subsection{Notation and nonsmooth preliminaries}

All vector norms are Euclidean, and $\ip{\bx}{\by}=\bx^\top\by$. For $r>0$ and $\bx\in\RR^d$, write
\[
\BB_r(\bx)=\{\by\in\RR^d:\norm{\by-\bx}\le r\},
\qquad
\BB_r^\circ(\bx)=\{\by\in\RR^d:\norm{\by-\bx}<r\}.
\]
The symbols $\sphere^{d-1}$, $\mathbf0$, $\mathbf e_j$, and $\mathbf I$ denote the unit sphere, the zero vector, the standard basis vectors, and the identity matrix, respectively. The orthogonal complement of $\bu$ is denoted by $\bu^\perp$. For a scalar $a$, write $a_+=\max\{a,0\}$. The symbols $\cv$, $\operatorname{int}$, and $\arg\min f$ denote convex hull, interior, and the set of global minimizers of $f$, respectively. For a finite set $S$, its cardinality is $|S|$. The notation $\Lip(f)$ denotes the least global Lipschitz constant. All logarithms are natural, and $\lfloor a\rfloor$ denotes the greatest integer not exceeding $a$. The notation $\exp(\Theta(d))$ means upper and lower exponential bounds with positive exponent constants independent of $d$.

For a locally Lipschitz function $f:\RR^d\to\RR$, a point $\bx\in\RR^d$, and a direction $\bv\in\RR^d$, its Clarke generalized directional derivative $f^\circ(\bx;\bv)$ and Clarke subdifferential $\partial f(\bx)$ are defined by
\begin{equation}
\begin{aligned}
f^\circ(\bx;\bv)
&=
\limsup_{\substack{\by\to\bx\\t\downarrow0}}
\frac{f(\by+t\bv)-f(\by)}{t},\\
\partial f(\bx)
&=
\{\bg\in\RR^d:
\ip{\bg}{\bv}\le f^\circ(\bx;\bv)
\text{ for every }\bv\in\RR^d\}.
\end{aligned}
\label{eq:clarke-definition}
\end{equation}
An unsubscripted $\partial$ always means this Clarke subdifferential.

The following three consequences of \eqref{eq:clarke-definition} are used in the construction. If $f=h$ on an open neighborhood of $\bx$, then their difference quotients coincide locally, giving
\begin{equation}
f^\circ(\bx;\bv)=h^\circ(\bx;\bv)
\quad\text{for every }\bv,
\qquad
\partial f(\bx)=\partial h(\bx).
\label{eq:locality}
\end{equation}
For $c>0$, positive scaling commutes with the limsup and the subdifferential inequalities, giving
\begin{equation}
\partial(cf)(\bx)=c\,\partial f(\bx).
\label{eq:positive-scaling}
\end{equation}
For an affine function $h(\bx)=\ip{\ba}{\bx}+b$, where $\ba\in\RR^d$ and $b\in\RR$,
\begin{equation}
\partial h(\bx)=\{\ba\}.
\label{eq:affine-subdifferential}
\end{equation}

For $\delta>0$, define the Goldstein subdifferential by
\begin{equation}
\partial_\delta f(\bx)
=
\cv\left(
\bigcup_{\by\in\BB_\delta(\bx)}\partial f(\by)
\right).
\label{eq:goldstein-definition}
\end{equation}
For $\epsilon>0$, a point $\bx$ is $(\delta,\epsilon)$-Goldstein stationary if
\begin{equation}
\dist(\mathbf0,\partial_\delta f(\bx))
:=
\inf_{\bg\in\partial_\delta f(\bx)}\norm{\bg}
\le\epsilon.
\label{eq:stationarity}
\end{equation}
Thus, $\delta$ is a neighborhood radius, while $\epsilon$ bounds the distance from the origin to $\partial_\delta f(\bx)$. We follow the parameter order of Jordan et al.\ \cite{DNN}. The same condition is called $(\epsilon,\delta)$-Goldstein approximate stationarity, or $(\epsilon,\delta)$-GAS, by Tian and So \cite[Definition~3]{nogas}. The upper-bound argument will use a minimum-norm element of this set.

\begin{lemma}[Compactness of the Goldstein subdifferential]
\label{lem:compactness}
If $f$ is globally $L$-Lipschitz, then $\partial_\delta f(\bx)$ is nonempty, compact, and convex, and is contained in $\BB_L(\mathbf0)$, for every $\bx\in\RR^d$ and $\delta>0$.
\end{lemma}

\begin{proof}
The Clarke subdifferential is nonempty and compact and lies in $\BB_L(\mathbf0)$ \cite[Proposition~2.1.2]{Clarke1990}. The graph over $\BB_\delta(\bx)$ is a closed subset of the compact product $\BB_\delta(\bx)\times\BB_L(\mathbf0)$ and is therefore compact. Projecting onto the subgradient coordinate shows that $\bigcup_{\by\in\BB_\delta(\bx)}\partial f(\by)$ is nonempty and compact. In finite dimensions, its convex hull is compact and remains in $\BB_L(\mathbf0)$. See also \cite[Section~2]{Gebken2025} and \cite[Proposition~2.3 and Corollary~2.13]{MankauSchuricht2019}.
\end{proof}

Consequently, the infimum in \eqref{eq:stationarity} is attained, and no additional closure is needed in \eqref{eq:goldstein-definition}.

\subsection{Function class, algorithms, and query counts}

Fix a prescribed initial point $\bx_0\in\RR^d$. For $L,\Delta>0$, let
\begin{equation}
\mathcal F_d(L,\Delta;\bx_0)
=
\left\{
f:\RR^d\to\RR:
\begin{array}{l}
|f(\bx)-f(\by)|\le L\norm{\bx-\by}
\quad\text{for all }\bx,\by,\\
\inf_{\bx\in\RR^d}f(\bx)>-\infty,\\
f(\bx_0)-\inf_{\bx\in\RR^d}f(\bx)\le\Delta
\end{array}
\right\}.
\label{eq:function-class}
\end{equation}
Thus $L$ is an upper bound, not necessarily the least Lipschitz constant. No convexity, smoothness, weak convexity, or representation-size assumption is imposed.

A query is one call to an oracle at a specified point. The algorithm receives
\[
d,\quad L,\quad\Delta,\quad\bx_0,\quad\delta,\quad\epsilon.
\]
Its first oracle call is at $\bx_0$. Following the local-oracle definition in Tian and So \cite[Section~2.3 and Remark~1]{nogas}, an admissible response to a query at $\bq$ is any locally Lipschitz function $h_{\bq}:\RR^d\to\RR$ satisfying
\begin{equation}
\mathcal O_f(\bq)=h_{\bq},\qquad
h_{\bq}(\by)=f(\by)\quad
\text{for all }\by\in\BB_{\nu_{\bq}}(\bq),
\qquad \nu_{\bq}>0.
\label{eq:oracle}
\end{equation}
The radius $\nu_{\bq}$ is not returned. Two oracle responses provide the same local information at $\bq$ if they agree on an open neighborhood of $\bq$. Neither response specifies a ball of known radius on which agreement with the objective is guaranteed.

A deterministic algorithm in this model may use all local information supplied by the oracle. Its internal computation, next query, stopping decision, and output depend only on the input and the local information in the previous oracle responses. In particular, replacing any finite collection of responses by functions that agree with them near the corresponding query points cannot change the subsequent computation before a new response is received. The neighborhoods need not have the same radius: the minimum of finitely many positive radii gives the common-radius condition in \cite{nogas}. Information encoded solely outside these neighborhoods is not available to the algorithm.

Query locations are unrestricted in $\RR^d$, repeated queries are allowed, and the output need not be a queried point. The pair $(f(\bq),\partial f(\bq))$ depends only on the local information at $\bq$. An algorithm using function values and the entire Clarke subdifferential is therefore a special case. The same applies to local derivative information whenever it is defined. For a fixed objective, replacing responses by locally agreeing functions induces the same execution. The first-order upper bound in Section~\ref{sec:coarse} is stated for an oracle returning a function value and one Clarke subgradient per call; the zeroth-order upper bound uses only function values.

Every oracle call is counted, including the initial one. Arbitrary finite internal computation with exact real information is allowed between calls; in particular, real arithmetic and comparisons are exact. An execution that computes forever without returning fails the finite-return requirement; no procedure for detecting such behavior is assumed. These are query-complexity conventions, not bounds on bit complexity or total running time. Finite preprocessing is allowed without any bound on its cost.

For a deterministic algorithm $A$, a valid output is a $(\delta,\epsilon)$-Goldstein stationary point. If the execution returns such a point in finite time, let $\tau_A(f;\bx_0,\delta,\epsilon)$ be its total number of calls. If it never returns or returns an invalid point, set
\[
\tau_A(f;\bx_0,\delta,\epsilon)=+\infty.
\]
Dependence on the fixed input parameters $d,L,\Delta$ is suppressed in this notation. The optimal worst-case query complexity in this local model is
\begin{equation}
\mathcal{T}_{\rm det}(d,L,\Delta,\delta,\epsilon;\bx_0)
=
\inf_A\
\sup_{f\in\mathcal F_d(L,\Delta;\bx_0)}
\tau_A(f;\bx_0,\delta,\epsilon).
\label{eq:complexity}
\end{equation}
The infimum is over deterministic algorithms in the model above. A \emph{uniform $T$-call guarantee} means that the same algorithm satisfies $\tau_A(f;\bx_0,\delta,\epsilon)\le T$ for every admissible $f$. Eventual success after more than $T$ calls gives a finite value $\tau_A>T$, not $+\infty$. Receiving the $T$th response still permits a return within budget; requesting another call does not. Appendix~\ref{app:model} treats algorithm-chosen initial points and optional pre-query returns.

\subsection{Main theorems}

The lower bound uses the local oracle defined above, whereas the upper bound needs only values. Their comparison at fixed accuracy determines the exponential order in $d$.

\Needspace{18\baselineskip}
\begin{theorem}[Accuracy-dependent query complexity lower bound]
\label{thm:main}
Let $d\ge3$, $L,\Delta>0$, $\bx_0\in\RR^d$, and
\[
0<\delta<\frac{\Delta}{L},\qquad 0<\epsilon<\frac L3.
\]
For every integer $T\ge1$ satisfying
\begin{equation}
T+1\le\left(\frac{L-\epsilon}{2\epsilon}\right)^{(d-2)/2},
\label{eq:budget}
\end{equation}
and every deterministic algorithm $A$ in the local-oracle model above, there exists a fixed $f\in\mathcal F_d(L,\Delta;\bx_0)$ with $\tau_A(f;\bx_0,\delta,\epsilon)>T$. Every point actually queried among the first $T$ calls, and the output if $A$ returns within $T$ calls, satisfies
\[
\dist(\mathbf0,\partial_\delta f(\bq))>\epsilon.
\]
Moreover, $f(\bx_0)=0$, $\inf f=-\Delta$, and $\interior(\arg\min f)\ne\varnothing$.
\end{theorem}

The order of quantifiers is $\forall A\,\exists f$, with the input and budget fixed. A function $f$ witnessing failure of the algorithm within the budget is called a \emph{hard instance}. It is allowed to depend on the algorithm as well as the fixed inputs and budget. Sections~\ref{sec:construction}--\ref{sec:replay} give the construction and proof.

\Needspace{14\baselineskip}
\begin{theorem}[Deterministic zeroth-order query bound]
\label{thm:upper}
Let $d\ge1$, $L,\Delta,\delta>0$, $\bx_0\in\RR^d$, and $0<\epsilon<L$. A deterministic algorithm using only function values returns a $(\delta,\epsilon)$-Goldstein stationary point of every $f\in\mathcal F_d(L,\Delta;\bx_0)$ within
\begin{equation}
1+\left(\frac{2\Delta}{\delta\epsilon}+1\right)
\left(1+\frac{4L}{\epsilon}\right)^d
\label{eq:upper-bound}
\end{equation}
total calls. If $\epsilon\ge L$, the initial point is already valid and the prescribed initial call suffices.
\end{theorem}

Since $h_{\bq}(\bq)=f(\bq)$, this also bounds \eqref{eq:complexity}. Both bounds hold when each oracle call returns the function value together with either the entire Clarke subdifferential or one Clarke subgradient: the upper bound ignores subgradients, and the lower-bound construction has singleton subdifferentials at every query used in the construction.

\Needspace{13\baselineskip}
\begin{corollary}[Fixed-accuracy exponential dimension dependence]
\label{cor:theta-d}
For every $d\ge4$ and $\bx_0\in\RR^d$,
\begin{equation*}
\left\lfloor\left(\frac72\right)^{(d-2)/2}\right\rfloor
\le\mathcal T_{\rm det}\left(d,1,1,\tfrac12,\tfrac18;\bx_0\right)
\le1+33^{d+1}.
\end{equation*}
Consequently, the fixed-accuracy complexity is $\exp(\Theta(d))$. There is no uniform deterministic bound that is a joint polynomial in $(d,\delta^{-1},\epsilon^{-1})$ for the full function class, even with $L=\Delta=1$.
\end{corollary}

\begin{proof}
Set
\[
E_d=\left(\frac72\right)^{(d-2)/2},\qquad T_d=\lfloor E_d\rfloor-1.
\]
For $d\ge4$, $T_d\ge1$ and $T_d+1\le E_d$. Theorem~\ref{thm:main}, applied with $L=\Delta=1$, $\delta=1/2$, and $\epsilon=1/8$, gives for every $A$ an admissible $f$ with $\tau_A(f)>T_d$. Since finite query counts are integers, taking the supremum over $f$ and then the infimum over $A$ gives the lower bound $\lfloor E_d\rfloor$.

At the same parameters, \eqref{eq:upper-bound} gives
\[
1+(32+1)33^d=1+33^{d+1}.
\]
The inequality $\lfloor E_d\rfloor\ge E_d/2$ then gives
\[
\frac{d-2}{2}\log\frac72-\log2
\le\log\mathcal T_{\rm det}\left(d,1,1,\tfrac12,\tfrac18;\bx_0\right)
\le(d+1)\log33+\log2.
\]
This proves the exponential order. A joint polynomial guarantee in $(d,\delta^{-1},\epsilon^{-1})$ would reduce to a polynomial in $d$ at $(\delta,\epsilon)=(1/2,1/8)$ and contradict the lower bound for sufficiently large $d$.
\end{proof}

For general parameters with $0<\epsilon<L/3$ and $0<\delta<\Delta/L$, the same integer-budget argument gives
\begin{equation*}
\mathcal T_{\rm det}(d,L,\Delta,\delta,\epsilon;\bx_0)
\ge\left\lfloor\left(\frac{L-\epsilon}{2\epsilon}\right)^{(d-2)/2}\right\rfloor
\qquad(d\ge3).
\end{equation*}
When the floor is at least two, apply Theorem~\ref{thm:main} with a budget one smaller; when it is one, the initial call gives the bound. Together with Theorem~\ref{thm:upper}, this yields $\exp(\Theta(d))$ for fixed admissible parameters.

\section{A finite-set construction}
\label{sec:construction}

The construction below assigns the prescribed affine local data to a finite point set satisfying a cardinality condition. Its scalar parameters remain free until Section~\ref{sec:replay}.

\subsection{Overview of the construction}
\label{sec:construction-outline}

The main difficulty is to make a finite query history consistent with one objective while excluding Goldstein stationarity at every relevant point. Affine agreement near the queries fixes the oracle responses, but it is not enough: the neighborhood $\BB_\delta(\bq)$ used in the stationarity condition can extend well beyond the region of affine agreement at $\bq$. The construction therefore uses an interpolant whose subgradients share a positive projection in one direction, even where its local formulas change.

Let $S=\{\bq_1,\ldots,\bq_N\}\subset\RR^d$ be a finite set of $N$ distinct points at which both local affine agreement and Goldstein nonstationarity are required. Choose a unit vector $\bu$ with first coordinate $\alpha>0$ and with
\[
|\langle\bu,\bq_i-\bq_j\rangle|\le\gamma\norm{\bq_i-\bq_j}
\qquad(i\ne j),
\]
where $0<\alpha<\gamma<1$. A Gaussian estimate and a union bound allow $N$ to be exponential in $d$. This bound on inner products with the pairwise differences replaces the exact orthogonality to query-dependent subspaces used in \cite{DNN,nogas}.

Let $\mathbf P=\mathbf I-\bu\bu^\top$ be the orthogonal projection onto $\bu^\perp$. For $N\ge2$, the interpolant is
\[
F(\bx)=\min_{1\le i\le N}\left\{
\mathbf e_1^\top(\bx-\bq_i)+\beta\bigl(\norm{\mathbf P(\bx-\bq_i)}-r\bigr)_+
\right\},
\]
where $\beta>1/\sqrt{1-\gamma^2}$ is the penalty coefficient and $r>0$ is chosen in \eqref{eq:parameters}. Each indexed function is a branch. The norm in its penalty is the distance to the line through $\bq_i$ parallel to $\bu$; only the excess beyond $r$ is penalized. The bound on pairwise differences ensures that the $i$th branch is strictly smaller than the others near $\bq_i$, where its own penalty vanishes. Hence $F(\bx)=\mathbf e_1^\top(\bx-\bq_i)$ there. For a singleton, the affine function itself suffices.

Every penalty is unchanged by translation along $\bu$, which gives
\[
F(\bx+t\bu)=F(\bx)+\alpha t.
\]
The Clarke defining inequalities in the directions $\bu$ and $-\bu$ then give $\langle\bg,\bu\rangle=\alpha$ for every $\bg\in\partial F(\bx)$. This common projection survives branch switches and convex combinations of subgradients. The bound $L_F$ from Lemma~\ref{lem:interpolation} need not be the least Lipschitz constant of $F$. Scaling by $L/L_F$ and truncating below $-\Delta$ produce the required objective. The condition $\delta<\Delta/L$ keeps the truncation inactive on an open set containing each ball $\BB_\delta(\bq_i)$. The remaining task is to choose scalar parameters satisfying $L\alpha/L_F>\epsilon$ while the bound on pairwise differences still permits the desired number of points.

These scalar parameters are fixed before any queries are generated. Each query at $\bq$ receives the affine function $(L/L_F)\mathbf e_1^\top(\bx-\bq)$, and only the answered part of the history is retained. If the algorithm returns within the budget, its output is added to $S$, whether or not it was queried. The finite-set construction then supplies one objective realizing every retained response. When run on this fixed objective, the algorithm reproduces the retained history and the action following it by determinism and locality. Section~\ref{sec:replay} gives this resisting-oracle argument, including the case of an unqueried output.

\subsection{Choice of direction}

The direction $\bu$ must satisfy two competing requirements: its first coordinate must remain positive, while projection onto $\bu^\perp$ must retain enough of each pairwise difference to separate the affine branches. For $0<\alpha<\gamma<1$, write
\[
s_\alpha:=\sqrt{1-\alpha^2},\qquad c_\gamma:=\sqrt{1-\gamma^2}.
\]
The estimate below ensures that orthogonal projection preserves at least the fraction $c_\gamma$ of the distance between any two points.

\Needspace{15\baselineskip}
\begin{lemma}[A direction with bounded inner products]
\label{lem:direction}
Let $d\ge3$, $0<\alpha<\gamma<1$, and $S=\{\bq_1,\ldots,\bq_N\}\subset\RR^d$ be a nonempty finite set of distinct points. If
\begin{equation}
N(N-1)\left(\frac{1-\gamma^2}{1-\alpha^2}\right)^{(d-2)/2}<1,
\label{eq:cardinality}
\end{equation}
there is a unit vector $\bu\in\RR^d$ such that
\begin{equation}
\ip{\bu}{\mathbf e_1}=\alpha,\qquad
|\ip{\bu}{\bq_i-\bq_j}|\le\gamma\norm{\bq_i-\bq_j}\quad(i\ne j).
\label{eq:avoidance}
\end{equation}
For $N=1$, condition \eqref{eq:cardinality} holds automatically and the second requirement is empty.
\end{lemma}

\begin{proof}
For $N=1$, the choice $\bu=\alpha\mathbf e_1+s_\alpha\mathbf e_2$ suffices. For the rest of the proof, assume $N\ge2$. The symbols $\PP$ and $\EE$ denote probability and expectation. Choose $\bv$ uniformly on the unit sphere in $\mathbf e_1^\perp$ and set
\begin{equation}
\bu=\alpha\mathbf e_1+s_\alpha\bv.
\label{eq:u}
\end{equation}
Then $\norm{\bu}=1$ and $\ip{\bu}{\mathbf e_1}=\alpha$. For a fixed unit vector $\bw$, define the set of choices that violate the inner-product bound by
\[
\mathcal E(\bw):=\left\{\bv\in\mathbf e_1^\perp:
\norm{\bv}=1,\ \bigl|\ip{\alpha\mathbf e_1+s_\alpha\bv}{\bw}\bigr|>\gamma\right\}.
\]

Write $\bw=a\mathbf e_1+b\boldsymbol\omega$, where $a^2+b^2=1$, $b\ge0$, and, when $b>0$, $\boldsymbol\omega$ is a unit vector in $\mathbf e_1^\perp$. If $b=0$, then $|\ip{\bu}{\bw}|=\alpha<\gamma$ and $\mathcal E(\bw)$ is empty. For $b>0$, Cauchy--Schwarz gives
\[
|\ip{\bu}{\bw}|
=|\alpha a+s_\alpha b\ip{\bv}{\boldsymbol\omega}|
\le\sqrt{\alpha^2+s_\alpha^2|\ip{\bv}{\boldsymbol\omega}|^2}.
\]
Setting $\xi=\sqrt{(\gamma^2-\alpha^2)/(1-\alpha^2)}\in(0,1)$ gives
\begin{equation}
\mathcal E(\bw)\subseteq\left\{\bv\in\mathbf e_1^\perp:
\norm{\bv}=1,\ |\ip{\bv}{\boldsymbol\omega}|>\xi\right\}.
\label{eq:bad-event}
\end{equation}

Identify $\mathbf e_1^\perp$ with $\RR^n$, where $n=d-1$, using $\boldsymbol\omega$ as the first coordinate vector. Let $\mathbf G=(G_1,\ldots,G_n)$ have independent standard normal coordinates. Rotational invariance allows the representation $\bv=\mathbf G/\norm{\mathbf G}$, with $\mathbf G\ne\mathbf0$ almost surely. Put
\[
Z=\sum_{j=2}^nG_j^2,\qquad \lambda_\xi=\frac{\xi^2}{1-\xi^2}.
\]
Then $G_1$ and $Z$ are independent, and
\[
\ip{\bv}{\boldsymbol\omega}=\frac{G_1}{\sqrt{G_1^2+Z}},\qquad
\PP\bigl(|\ip{\bv}{\boldsymbol\omega}|>\xi\bigr)
=\PP\left(\frac{G_1^2}{G_1^2+Z}>\xi^2\right)
=\PP(G_1^2>\lambda_\xi Z).
\]
For $s>0$, Markov's inequality applied to $e^{sG_1}$ gives
\[
\PP(G_1>s)=\PP(e^{sG_1}>e^{s^2})
\le e^{-s^2}\EE e^{sG_1}=e^{-s^2/2}.
\]
By symmetry, and also at $s=0$,
\begin{equation}
\PP(|G_1|>s)\le2e^{-s^2/2}\qquad(s\ge0).
\label{eq:gaussian-tail}
\end{equation}
Conditioning on $Z$ and using independence, followed by \eqref{eq:gaussian-tail}, yields
\begin{align*}
\PP(G_1^2>\lambda_\xi Z)
&=\EE\left[\PP\left(|G_1|>\sqrt{\lambda_\xi Z}\mid Z\right)\right]\\
&\le2\EE e^{-\lambda_\xi Z/2}
=2\prod_{j=2}^n\EE e^{-\lambda_\xi G_j^2/2}\\
&=2(1+\lambda_\xi)^{-(n-1)/2}
=2(1-\xi^2)^{(d-2)/2}.
\end{align*}
The first equality is the tower property; conditioning on $Z$ leaves the law of $G_1$ unchanged. Each factor is given by the Gaussian integral
\[
\EE e^{-\lambda_\xi G_j^2/2}
=\frac1{\sqrt{2\pi}}\int_{-\infty}^{\infty}e^{-(1+\lambda_\xi)t^2/2}\,\df t
=(1+\lambda_\xi)^{-1/2}.
\]
Since $1-\xi^2=(1-\gamma^2)/(1-\alpha^2)$, \eqref{eq:bad-event} gives, including the case $b=0$,
\[
\PP\bigl(\bv\in\mathcal E(\bw)\bigr)
\le2\left(\frac{1-\gamma^2}{1-\alpha^2}\right)^{(d-2)/2}.
\]

For $i<j$, let $\bw_{ij}=(\bq_i-\bq_j)/\norm{\bq_i-\bq_j}$. The absolute-value condition is unchanged when the two points are interchanged. It is therefore enough to consider the $N(N-1)/2$ unordered pairs. The union bound gives
\[
\PP\left(\bv\in\bigcup_{i<j}\mathcal E(\bw_{ij})\right)
\le\sum_{i<j}\PP(\bv\in\mathcal E(\bw_{ij}))
\le N(N-1)\left(\frac{1-\gamma^2}{1-\alpha^2}\right)^{(d-2)/2}<1.
\]
A choice outside this union in \eqref{eq:u} gives \eqref{eq:avoidance}.
\end{proof}

No independence between the bad events is needed. Probability is used only to choose a fixed direction for a fixed set, not to randomize the optimization algorithm.

\subsection{Construction of an auxiliary Lipschitz function}

To preserve a common slope along $\bu$, each affine model is penalized only in the transverse directions. The penalty must make a branch uniquely active near its center without increasing the Lipschitz bound with the number of centers.

\Needspace{19\baselineskip}
\begin{lemma}[A Lipschitz function with prescribed local behavior]
\label{lem:interpolation}
Let $S=\{\bq_1,\ldots,\bq_N\}\subset\RR^d$ be a nonempty finite set of distinct points, and let $\bu$ satisfy \eqref{eq:avoidance} for $0<\alpha<\gamma<1$. Choose $\beta>1/c_\gamma$, and define the unscaled Lipschitz bound
\begin{equation}
L_F:=\sqrt{\alpha^2+(s_\alpha+\beta)^2},\qquad
s_\alpha=\sqrt{1-\alpha^2},\quad c_\gamma=\sqrt{1-\gamma^2}.
\label{eq:LF}
\end{equation}
There is a globally $L_F$-Lipschitz function $F:\RR^d\to\RR$ with the following properties:
\begin{enumerate}[(i)]
\item $F(\bx)=\mathbf e_1^\top(\bx-\bq_i)$ on an open neighborhood of every $\bq_i$. Consequently, $F(\bq_i)=0$ and $\partial F(\bq_i)=\{\mathbf e_1\}$.
\item For every $\bx\in\RR^d$ and $t\in\RR$,
\begin{equation}
F(\bx+t\bu)=F(\bx)+\alpha t.
\label{eq:translation}
\end{equation}
\item For every $\bx\in\RR^d$ and $\bg\in\partial F(\bx)$,
\begin{equation}
\ip{\bg}{\bu}=\alpha.
\label{eq:clarke-projection}
\end{equation}
\end{enumerate}
For $N\ge2$, the neighborhood in (i) can be the open ball $\BB_r^\circ(\bq_i)$ for the common radius in \eqref{eq:parameters}.
\end{lemma}

\begin{proof}
For $N=1$, take $F(\bx)=\mathbf e_1^\top(\bx-\bq_1)$. All claims follow from $\ip{\bu}{\mathbf e_1}=\alpha$ and $L_F^2=1+2s_\alpha\beta+\beta^2\ge1$. For $N\ge2$, set
\begin{equation}
\mathbf P=\mathbf I-\bu\bu^\top,\qquad
s_*:=\min_{i\ne j}\norm{\bq_i-\bq_j}>0,\qquad
r:=\frac{\beta c_\gamma-1}{2\beta}s_*>0.
\label{eq:parameters}
\end{equation}
Since $\mathbf P$ is the orthogonal projection onto $\bu^\perp$,
\begin{equation}
\norm{\mathbf P\bw}^2=\norm{\bw}^2-|\ip{\bu}{\bw}|^2,\qquad
\norm{\mathbf P\bw}\le\norm{\bw}.
\label{eq:projection-identities}
\end{equation}
Define the branches by
\begin{equation}
F_i(\bx)=\mathbf e_1^\top(\bx-\bq_i)
+\beta\bigl(\norm{\mathbf P(\bx-\bq_i)}-r\bigr)_+,
\label{eq:branch}
\end{equation}
and let
\begin{equation}
\begin{aligned}
F(\bx)&=\min_{1\le i\le N}F_i(\bx)\\
&=\mathbf e_1^\top\bx+\min_{1\le i\le N}
\left\{-\mathbf e_1^\top\bq_i+\beta\max\{0,\norm{\mathbf P(\bx-\bq_i)}-r\}\right\}.
\end{aligned}
\label{eq:F}
\end{equation}

For the Lipschitz estimate, write $\mathbf e_1=\alpha\bu+\mathbf P\mathbf e_1$, with $\norm{\mathbf P\mathbf e_1}=s_\alpha$. The norm and positive-part maps are $1$-Lipschitz. For $\mathbf a=\bx-\by$, they give
\begin{align*}
|F_i(\bx)-F_i(\by)|
&\le\alpha|\ip{\bu}{\mathbf a}|+(s_\alpha+\beta)\norm{\mathbf P\mathbf a}\\
&\le\sqrt{\alpha^2+(s_\alpha+\beta)^2}
\sqrt{|\ip{\bu}{\mathbf a}|^2+\norm{\mathbf P\mathbf a}^2}
=L_F\norm{\bx-\by}.
\end{align*}
If $F(\by)=F_i(\by)$, then $F(\bx)-F(\by)\le F_i(\bx)-F_i(\by)\le L_F\norm{\bx-\by}$. Interchanging $\bx$ and $\by$ proves the same bound for $F$, with no factor depending on $N$.

\medskip
\noindent\emph{Local agreement.}\enspace Equations \eqref{eq:avoidance} and \eqref{eq:projection-identities} imply
\begin{equation}
\norm{\mathbf P(\bq_i-\bq_j)}\ge c_\gamma\norm{\bq_i-\bq_j}
\qquad(i\ne j).
\label{eq:projected-distance}
\end{equation}
Fix $i$ and write $\bx=\bq_i+\mathbf h$, where $\norm{\mathbf h}<r$. The $i$th penalty vanishes, giving $F_i(\bx)=\mathbf e_1^\top\mathbf h$. For $j\ne i$, put $\bw=\bq_i-\bq_j$. By \eqref{eq:projected-distance},
\begin{align*}
F_j(\bq_i+\mathbf h)-F_i(\bq_i+\mathbf h)
&=\mathbf e_1^\top\bw+\beta(\norm{\mathbf P(\bw+\mathbf h)}-r)_+\\
&\ge-\norm{\bw}+\beta\norm{\mathbf P\bw}-\beta\norm{\mathbf P\mathbf h}-\beta r\\
&\ge(\beta c_\gamma-1)s_*-\beta\norm{\mathbf h}-\beta r\\
&=\beta(r-\norm{\mathbf h})>0.
\end{align*}
Thus $F(\bx)=F_i(\bx)=\mathbf e_1^\top(\bx-\bq_i)$ on $\BB_r^\circ(\bq_i)$. At the center, $F(\bq_i)=0$, and locality gives $\partial F(\bq_i)=\{\mathbf e_1\}$.

\medskip
\noindent\emph{Common projection.}\enspace Since $\mathbf P\bu=\mathbf0$, translation along $\bu$ leaves each penalty unchanged. Hence,
\[
F_i(\bx+t\bu)=F_i(\bx)+\alpha t,\qquad
F(\bx+t\bu)=F(\bx)+\alpha t
\qquad(t\in\RR).
\]
The difference quotients in the directions $\bu$ and $-\bu$ are identically $\alpha$ and $-\alpha$, respectively. Therefore
\[
F^\circ(\bx;\bu)=\alpha,\qquad F^\circ(\bx;-\bu)=-\alpha.
\]
For $\bg\in\partial F(\bx)$, the defining inequalities give
\[
\ip{\bg}{\bu}\le\alpha,\qquad -\ip{\bg}{\bu}\le-\alpha,
\]
which prove \eqref{eq:clarke-projection}. This argument applies also where branches switch or a penalty is not differentiable; it requires no equality formula for the Clarke subdifferential of a minimum.
\end{proof}

\subsection{Scaling and lower truncation}

The interpolant is unbounded below along $-\bu$. Scaling to the prescribed Lipschitz bound and truncating at $-\Delta$ gives a bounded-below objective. The condition $\delta<\Delta/L$ leaves the truncation inactive on an open set containing each closed ball $\BB_\delta(\bq_i)$.

\Needspace{20\baselineskip}
\begin{proposition}[Consistency with affine oracle responses]
\label{prop:transcript}
Let $d\ge3$ and $S=\{\bq_1,\ldots,\bq_N\}\subset\RR^d$ be a nonempty finite set of distinct points. Fix $0<\alpha<\gamma<1$ satisfying \eqref{eq:cardinality}, choose $\beta>1/c_\gamma$, and define $L_F$ by \eqref{eq:LF}. Given $L,\Delta>0$ and $0<\delta<\Delta/L$, there exist a fixed function $f:\RR^d\to\RR$, a unit vector $\bu$, and a common radius $\rho>0$ such that:
\begin{enumerate}[(i)]
\item
$\Lip(f)\le L$, $\inf f=-\Delta$, and $\interior(\arg\min f)\ne\varnothing$;

\item
for every $i$ and $\bx\in\BB_\rho^\circ(\bq_i)$,
\begin{equation}
f(\bx)=\frac{L}{L_F}\mathbf e_1^\top(\bx-\bq_i),
\qquad
\bigl(f(\bq_i),\partial f(\bq_i)\bigr)
=
(0,\{L\mathbf e_1/L_F\});
\label{eq:local-copy}
\end{equation}

\item
for every $i$ and $\bg\in\partial_\delta f(\bq_i)$,
\begin{equation}
\ip{\bg}{\bu}=\frac{L\alpha}{L_F},
\label{eq:goldstein-projection}
\end{equation}
and consequently
\begin{equation}
\dist(\mathbf0,\partial_\delta f(\bq_i))\ge\frac{L\alpha}{L_F}.
\label{eq:goldstein-lower}
\end{equation}
\end{enumerate}
If $\bx_0\in S$, then $f\in\mathcal F_d(L,\Delta;\bx_0)$, $f(\bx_0)=0$, and its initial gap is exactly $\Delta$.
\end{proposition}

\begin{proof}
Choose $\bu$ by Lemma~\ref{lem:direction} and construct $F$ by Lemma~\ref{lem:interpolation}. Set
\begin{equation}
\phi(\bx)=\frac{L}{L_F}F(\bx),
\qquad
f(\bx)=\max\{\phi(\bx),-\Delta\}.
\label{eq:truncation}
\end{equation}
Both $\phi$ and $f$ are globally $L$-Lipschitz, since the scalar truncation map is $1$-Lipschitz. Also $f\ge-\Delta$.

For $N\ge2$, take $\rho=\min\{r,\Delta/(2L)\}$; for $N=1$, take $\rho=\Delta/(2L)$. On $\BB_\rho^\circ(\bq_i)$, the interpolant has its prescribed affine form and
\[
\phi(\bx)\ge-L\norm{\bx-\bq_i}>-\Delta/2.
\]
Thus $f=\phi$ there. Locality and \eqref{eq:affine-subdifferential} give \eqref{eq:local-copy}.

The translation identity \eqref{eq:translation} gives
\[
\phi(\bq_i-t\bu)=-\frac{L\alpha}{L_F}t.
\]
For $t>\Delta L_F/(L\alpha)$, this is strictly below $-\Delta$. By continuity, $\phi<-\Delta$ on an open neighborhood of such a point, and $f$ is constant there. Consequently $\inf f=-\Delta$ and $\operatorname{int}(\arg\min f)\ne\varnothing$. Every point in this open set has zero Clarke subgradient and is Goldstein stationary for every $\delta>0$.

To prove (iii), fix $\by\in\BB_\delta(\bq_i)$. Since $\phi(\bq_i)=0$,
\[
\phi(\by)\ge-L\delta>-\Delta.
\]
Local agreement is also required at boundary points of the ball. With
\[
\tau=\frac{\Delta-L\delta}{2L}>0,
\]
any $\bz\in\BB_\tau^\circ(\by)$ satisfies
\[
\phi(\bz)\ge\phi(\by)-L\norm{\bz-\by}
>-L\delta-L\tau=-\frac{\Delta+L\delta}{2}>-\Delta.
\]
Hence $f=\phi$ on an open neighborhood of every $\by\in\BB_\delta(\bq_i)$. By locality \eqref{eq:locality}, scaling \eqref{eq:positive-scaling}, and Lemma~\ref{lem:interpolation},
\[
\partial f(\by)=\frac{L}{L_F}\partial F(\by),\qquad
\ip{\bg}{\bu}=\frac{L\alpha}{L_F}
\quad\text{for every }\bg\in\partial f(\by).
\]
The projection remains unchanged under convex combinations, proving \eqref{eq:goldstein-projection}. Since $\norm{\bu}=1$, Cauchy--Schwarz gives $\norm{\bg}\ge L\alpha/L_F$ throughout $\partial_\delta f(\bq_i)$ and thus \eqref{eq:goldstein-lower}. Finally, if $\bx_0\in S$, then $f(\bx_0)=0$ and the initial gap is exactly $\Delta$, proving membership in \eqref{eq:function-class}.
\end{proof}

The affine radius $r$ is not required to be comparable to $\delta$, even when query points are close. The ball $\BB_\delta(\bq_i)$ need not be contained in the corresponding affine neighborhood: the translation identity controls its subgradients regardless of which branches are active.

\section{The deterministic query complexity lower bound}
\label{sec:replay}

Increasing the penalty coefficient strengthens branch separation but raises $L_F$, reducing the projection after scaling. The parameter choice below satisfies the separation condition, the required projection bound, and the probability bound for all pairwise differences simultaneously. The scalar parameters must be fixed before the queries, because $L/L_F$ appears in every oracle response.

\Needspace{16\baselineskip}
\subsection{Choice of scalar parameters}
\begin{lemma}[Parameters for a prescribed budget]
\label{lem:parameter-choice}
Let $d\ge3$, $L>0$, and $0<\epsilon<L/3$. Let the integer $T\ge1$ satisfy \eqref{eq:budget}. There are scalars $0<\alpha<\gamma<1$ and $\beta>1/c_\gamma$, depending only on $d,L,\epsilon,T$, for which $L_F$ in \eqref{eq:LF} satisfies
\begin{equation}
\frac{L\alpha}{L_F}>\epsilon,\qquad
T(T+1)\left(\frac{1-\gamma^2}{1-\alpha^2}\right)^{(d-2)/2}<1.
\label{eq:chosen-parameters}
\end{equation}
\end{lemma}

\begin{proof}
Set $\theta=\epsilon/L$ and
\[
n_T=T+1,\qquad
\chi_T=[n_T(n_T-1)]^{1/(d-2)},\qquad
\vartheta_T=\frac1{1+2\chi_T}.
\]
Here $n_T$ allows for an unqueried output, and $n_T(n_T-1)$ accounts for the pairwise differences. Since $n_T\ge2$, it follows that $\chi_T>1$. The budget condition and $n_T(n_T-1)<n_T^2$ give
\[
\chi_T<\frac{1-\theta}{2\theta},\qquad
\theta<\vartheta_T<\frac13.
\]
Choose $\vartheta=(\theta+\vartheta_T)/2$, so that $\theta<\vartheta<\vartheta_T$, and set
\begin{equation*}
\alpha=\sqrt{\frac{1+\vartheta}{2}},\qquad
s_\alpha=\sqrt{\frac{1-\vartheta}{2}},\qquad
\beta=\frac{s_\alpha}{\vartheta}.
\end{equation*}
Then $s_\alpha^2=1-\alpha^2$. The identity
\[
\vartheta^2\alpha^2+s_\alpha^2(1+\vartheta)^2
=\frac{1+\vartheta}{2}\bigl(\vartheta^2+(1-\vartheta)(1+\vartheta)\bigr)
=\alpha^2
\]
implies
\[
L_F^2=\alpha^2+(s_\alpha+\beta)^2=\frac{\alpha^2}{\vartheta^2},\qquad
\frac{L\alpha}{L_F}=L\vartheta>\epsilon.
\]
Also $\vartheta<\vartheta_T$ gives $\chi_T<s_\alpha^2/\vartheta$. Therefore
\[
0<\frac{\vartheta}{s_\alpha}<\frac{s_\alpha}{\chi_T}<s_\alpha.
\]
Define
\begin{equation*}
c_\gamma=\frac12\left(\frac{\vartheta}{s_\alpha}+\frac{s_\alpha}{\chi_T}\right),
\qquad \gamma=\sqrt{1-c_\gamma^2}.
\end{equation*}
These choices satisfy $0<c_\gamma<s_\alpha<1$ and therefore $0<\alpha<\gamma<1$. They also give $\beta c_\gamma>1$. Finally,
\[
T(T+1)\left(\frac{1-\gamma^2}{1-\alpha^2}\right)^{(d-2)/2}
=n_T(n_T-1)\left(\frac{c_\gamma}{s_\alpha}\right)^{d-2}
=\left(\frac{\chi_Tc_\gamma}{s_\alpha}\right)^{d-2}<1.
\]
\end{proof}

\subsection{A resisting-oracle argument}

With the affine responses prescribed, the finite-set construction gives one objective agreeing with the resulting finite history. Locality and determinism then reproduce that history on the fixed objective.

\begin{proof}[Proof of Theorem~\ref{thm:main}]
Fix the inputs, the budget $T$, and a deterministic algorithm $A$. Choose $\alpha,\gamma,\beta,L_F$ by Lemma~\ref{lem:parameter-choice}. First consider a run of $A$ with prescribed oracle responses, without yet choosing an objective. At each query point $\bq$, let the oracle return
\[
\mathcal O(\bq)=\ell_{\bq},\qquad
\ell_{\bq}(\bx)=\frac{L}{L_F}\mathbf e_1^\top(\bx-\bq).
\]
The first call is at $\bx_0$ and is counted. These responses determine the subsequent computation, although $\bu$ and the objective have not yet been chosen. Only a finite answered history will be realized, not an infinite sequence of prescribed responses. Retain the sequence of queries and responses up to a return, a request for call $T+1$, or internal computation that never produces another query or return. There are three cases.

\medskip
\noindent\emph{(i) Return within $T$ calls.}\enspace If $A$ returns $\widehat{\bx}$ after $1\le m\le T$ responses at $\bq^{(1)},\ldots,\bq^{(m)}$, set
\begin{equation}
S=\{\bq^{(1)},\ldots,\bq^{(m)},\widehat{\bx}\}.
\label{eq:protected-output}
\end{equation}
The output is fixed by the answered history, even if it was not queried, and $|S|\le m+1\le T+1$.

\medskip
\noindent\emph{(ii) Request for call $T+1$.}\enspace Retain the $T$ answered locations and set $S=\{\bq^{(1)},\ldots,\bq^{(T)}\}$. No response at the next requested point, and no inclusion of that point in $S$, is needed: reproducing the request already excludes a return within $T$ calls.

\medskip
\noindent\emph{(iii) No further query or return.}\enspace If, after $1\le m\le T$ responses, no further query or return is ever produced, set $S=\{\bq^{(1)},\ldots,\bq^{(m)}\}$.

These cases exhaust the possible behavior before call $T+1$; the argument does not require detecting nontermination. In every case $\bq^{(1)}=\bx_0\in S$ and $1\le N:=|S|\le T+1$. Repeated locations are removed only in forming the set, not from the query count.

\medskip
\noindent\emph{A fixed objective.}\enspace By \eqref{eq:chosen-parameters},
\[
N(N-1)\left(\frac{1-\gamma^2}{1-\alpha^2}\right)^{(d-2)/2}<1.
\]
Proposition~\ref{prop:transcript} therefore supplies $\bu$, $F$, and the fixed function $f(\bx)=\max\{LF(\bx)/L_F,-\Delta\}$ satisfying
\begin{equation}
f\in\mathcal F_d(L,\Delta;\bx_0),
\qquad
f(\bx_0)=0,
\qquad
\inf f=-\Delta,
\qquad
\interior(\arg\min f)\ne\varnothing.
\label{eq:hard-function-properties}
\end{equation}
For every $\bq\in S$,
\begin{equation}
\bigl(f(\bq),\partial f(\bq)\bigr)
=
\bigl(0,\{L\mathbf e_1/L_F\}\bigr),
\qquad
\dist(\mathbf0,\partial_\delta f(\bq))\ge L\alpha/L_F>\epsilon.
\label{eq:consistent}
\end{equation}
The local oracle requires neighborhood agreement: there is a common $\rho>0$ such that $f=\ell_{\bq}$ on $\BB_\rho^\circ(\bq)$ for every $\bq\in S$. Thus each prescribed response agrees with $f$ on the closed ball $\BB_{\rho/2}(\bq)$, as required by \eqref{eq:oracle}; the radius is not disclosed. Thus all retained oracle responses are consistent with this one fixed function. Fix an oracle that returns $\ell_{\bq}$ for $\bq\in S$ and $f$ itself elsewhere. All its responses are admissible, and both the objective and the representatives are now fixed.

\medskip
\noindent\emph{Consistency of the executions.}\enspace Run $A$ with this oracle and the same inputs. The first query and response are the same as in the run with prescribed responses. Inductively, once the two executions have received the same responses to the same queries, locality and determinism give identical internal computation and the same next action. If that action is another query in the retained history, its location lies in $S$ and receives the same affine response. If it is a return or the request for call $T+1$, that action is reproduced. If the internal computation never produces another action, it is identical in both executions. This proves agreement through the retained history and the behavior immediately following it.

In case (i), the returned point is invalid by \eqref{eq:protected-output} and \eqref{eq:consistent}; hence $\tau_A(f;\bx_0,\delta,\epsilon)=+\infty$. Including an unqueried output in $S$ imposes no additional call. In case (ii), any valid eventual return requires at least $T+1$ calls; otherwise $\tau_A=+\infty$. In case (iii), there is no finite return. Thus $\tau_A>T$ in every case. Every point actually queried among the first $T$ calls belongs to $S$ and therefore satisfies the strict nonstationarity bound in \eqref{eq:consistent}. The remaining properties are \eqref{eq:hard-function-properties}.
\end{proof}

Because all queried points in the retained history are nonstationary, the argument also excludes a guarantee that merely one of the first $T$ queried points is stationary, without requiring the algorithm to identify it.

\paragraph{Range of the parameter estimate.}
The value $1/3$ is a limit of this parameter estimate, not a threshold established for the optimization problem. Indeed, $\gamma>\alpha$ and $\beta>1/c_\gamma$ imply $\beta>1/s_\alpha$, whence
\[
\frac{\alpha^2}{L_F^2}<\frac{s_\alpha^2(1-s_\alpha^2)}{1+3s_\alpha^2}\le\frac19.
\]
The last inequality is equivalent to $(3s_\alpha^2-1)^2\ge0$. These estimates therefore cannot certify a projection of $L/3$ or larger.

\section{Deterministic upper bounds}
\label{sec:upper}

Goldstein nonstationarity implies descent over a segment of length $\delta$. Lipschitz continuity allows the descent direction to be approximated by a finite net, and the initial gap bounds the number of accepted moves. This gives a deterministic zeroth-order method. When $L/\sqrt2<\epsilon<L$, one returned subgradient suffices to choose the trial direction.

\subsection{A finite direction net}

An $\eta$-net of the unit sphere approximates every unit direction to distance at most $\eta$. The packing estimate below bounds its size; Appendix~\ref{app:net} gives a finite construction in the exact-real query model.

\begin{lemma}[A finite direction net]
\label{lem:net}
For every $d\ge1$ and $\eta>0$, there is a nonempty finite set $\mathcal V\subset\sphere^{d-1}$ such that
\begin{equation}
\text{for every }\bv\in\sphere^{d-1}
\text{ there is }\bw\in\mathcal V
\text{ with }\norm{\bv-\bw}\le\eta,
\label{eq:net-cover}
\end{equation}
and
\begin{equation}
|\mathcal V|\le(1+2/\eta)^d.
\label{eq:net-size}
\end{equation}
The set and its ordering can be constructed deterministically using finite internal computation, independently of $f$, in the exact-real model of Section~\ref{sec:model}.
\end{lemma}

\begin{proof}
For $\eta\ge2$, a fixed unit vector suffices. For $0<\eta<2$, let $\bw_1,\ldots,\bw_M$ be unit vectors with pairwise distances greater than $\eta$. The balls $\BB_{\eta/2}^\circ(\bw_j)$ are disjoint and contained in $\BB_{1+\eta/2}^\circ(\mathbf0)$. Comparing their volumes gives
\[
M(\eta/2)^d\le(1+\eta/2)^d,
\qquad M\le(1+2/\eta)^d.
\]
Starting with a fixed unit vector, append a point at distance greater than $\eta$ from all previously selected points whenever one exists. The packing bound forces this process to stop, yielding \eqref{eq:net-cover} with the size bound \eqref{eq:net-size}.

For a deterministic implementation, one must decide whether another point can be added. Appendix~\ref{app:net} gives a finite decision test and selects an admissible point from a fixed dense rational enumeration whenever the test succeeds. The insertion order gives the required ordering; no query to $f$ is used.
\end{proof}

\Needspace{14\baselineskip}
\subsection{A descent step of length \texorpdfstring{$\delta$}{delta}}

A separating direction controls every subgradient along the corresponding segment in $\BB_\delta(\bx)$. This permits a step of length $\delta$, rather than one whose admissible length depends on unknown local behavior.

\Needspace{11\baselineskip}
\begin{lemma}[Descent from Goldstein nonstationarity]
\label{lem:goldstein-descent}
Let $f:\RR^d\to\RR$ be globally $L$-Lipschitz, let $\delta>0$, and assume
\[
\dist(\mathbf0,\partial_\delta f(\bx))>\epsilon>0.
\]
Then there is a unit vector $\bv\in\RR^d$ such that
\begin{equation}
f(\bx-\delta\bv)<f(\bx)-\delta\epsilon.
\label{eq:full-radius-descent}
\end{equation}
\end{lemma}

\begin{proof}
Let $C=\partial_\delta f(\bx)$, which is nonempty, compact, and convex by Lemma~\ref{lem:compactness}, and choose a minimum-norm element $\bg_*\in C$. For $\bg\in C$ and $0<s\le1$, minimality gives
\[
0\le\norm{\bg_*+s(\bg-\bg_*)}^2-\norm{\bg_*}^2
=2s\ip{\bg_*}{\bg-\bg_*}+s^2\norm{\bg-\bg_*}^2.
\]
Dividing by $s$ and letting $s\downarrow0$ yields $\ip{\bg_*}{\bg-\bg_*}\ge0$. Since $\norm{\bg_*}>\epsilon$, the unit vector $\bv=\bg_*/\norm{\bg_*}$ satisfies
\[
\ip{\bg}{\bv}\ge\norm{\bg_*}>\epsilon\qquad(\bg\in C).
\]
Lebourg's mean value theorem \cite[Theorem~2.3.7]{Clarke1990} gives a point $\bz=\bx-\theta\delta\bv$, $\theta\in(0,1)$, and a vector $\bg\in\partial f(\bz)$ such that
\[
f(\bx-\delta\bv)-f(\bx)=-\delta\ip{\bg}{\bv}.
\]
Since $\bz\in\BB_\delta(\bx)$, the vector satisfies $\bg\in C$. The separation inequality therefore gives
\[
f(\bx-\delta\bv)-f(\bx)\le-\delta\norm{\bg_*}<-\delta\epsilon.
\]
\end{proof}

\subsection{Direction-net descent and query complexity}

Each iteration scans the same ordered net. Its radius is chosen to retain at least half of the decrease in Lemma~\ref{lem:goldstein-descent}. Values at accepted trial points are retained for the next iteration.

For $0<\epsilon<L$, set $\eta=\epsilon/(2L)$ and fix an ordered net $\mathcal V=(\bw_1,\ldots,\bw_M)$ from Lemma~\ref{lem:net}, with
\begin{equation}
M
\le
(1+2/\eta)^d
=
\left(1+\frac{4L}{\epsilon}\right)^d.
\label{eq:chosen-net-size}
\end{equation}
Algorithm~\ref{alg:net} queries every direction in the ordered net, then moves along the first one giving a decrease of at least $\delta\epsilon/2$. If no direction gives this decrease, it returns the current point.

\begin{algorithm}[H]
\caption{Deterministic direction-net descent}
\label{alg:net}
\begin{algorithmic}[1]
\REQUIRE
$L,\delta>0$, $0<\epsilon<L$, an initial point $\bx_0$.
\STATE
Construct and order an $\epsilon/(2L)$-net
$\mathcal V=(\bw_1,\ldots,\bw_M)$ of $\sphere^{d-1}$.
\STATE
Query $f(\bx_0)$; set $\bx\leftarrow\bx_0$ and
$a\leftarrow f(\bx_0)$.
\LOOP
\STATE
Query $b_j=f(\bx-\delta\bw_j)$ for every $j=1,\ldots,M$.
\STATE
Set $J=\{j:b_j\le a-\delta\epsilon/2\}$.
\IF{$J=\varnothing$}
\STATE
Return $\bx$.
\ELSE
\STATE
Set $j_*=\min J$.
\STATE
Set $\bx\leftarrow\bx-\delta\bw_{j_*}$ and
$a\leftarrow b_{j_*}$.
\ENDIF
\ENDLOOP
\end{algorithmic}
\end{algorithm}

\begin{proof}[Proof of Theorem~\ref{thm:upper}]
First, consider $0<\epsilon<L$. Lemma~\ref{lem:net} and Appendix~\ref{app:net} make the preprocessing finite and deterministic. The stored value is always $a=f(\bx)$, because accepted trial values are reused.

If $\bx$ is not $(\delta,\epsilon)$-Goldstein stationary, Lemma~\ref{lem:goldstein-descent} gives a unit vector $\bv$ satisfying \eqref{eq:full-radius-descent}. Choose $\bw_j$ with $\norm{\bw_j-\bv}\le\eta$. Then
\begin{align*}
f(\bx-\delta\bw_j)
&\le f(\bx-\delta\bv)+L\delta\norm{\bw_j-\bv}\\
&<f(\bx)-\delta\epsilon+L\delta\eta
=f(\bx)-\delta\epsilon/2.
\end{align*}
Thus $j\in J$. The stopping condition $J=\varnothing$ therefore certifies stationarity.

Each accepted move decreases the value by at least $\delta\epsilon/2$. Let $\bx^{(m)}$ denote the point reached after $m$ accepted moves. For every finite $m$,
\[
m\frac{\delta\epsilon}{2}
\le f(\bx_0)-f(\bx^{(m)})
\le f(\bx_0)-\inf f\le\Delta.
\]
Consequently,
\begin{equation}
m\le\frac{2\Delta}{\delta\epsilon}.
\label{eq:move-count}
\end{equation}
Since every scan is finite, the move bound \eqref{eq:move-count} forces termination. There are $m$ successful scans and one final unsuccessful scan. Counting the initial call and all $M$ trials in each scan, and using \eqref{eq:chosen-net-size}, gives
\begin{align*}
\tau_A(f;\bx_0,\delta,\epsilon)
&\le1+(m+1)M\\
&\le1+\left(\frac{2\Delta}{\delta\epsilon}+1\right)
\left(1+\frac{4L}{\epsilon}\right)^d.
\end{align*}
Repeated queries are charged again in this count; reusing stored values cannot increase it.

If $\epsilon\ge L$, Lemma~\ref{lem:compactness} gives $\dist(\mathbf0,\partial_\delta f(\bx_0))\le L\le\epsilon$. The prescribed initial call followed by a return therefore suffices.
\end{proof}

The separating direction is used only in the analysis. The algorithm compares values, and its stopping test is sufficient but not necessary for stationarity. An accepted move is not excluded at a stationary iterate and remains subject to the same decrease bound.

\subsection{A first-order bound for coarse accuracy}
\label{sec:coarse}

A single subgradient need not point in a useful descent direction. When $L/\sqrt2<\epsilon<L$, however, an unsuccessful trial certifies stationarity: the current subgradient and one subgradient on the trial segment have an average of sufficiently small norm. The oracle returns $f(\bx)$ and one $\bg\in\partial f(\bx)$. The guarantee holds for every valid subgradient selection. No effective extraction from an arbitrary local-function representation is assumed.

\begin{proposition}[Dimension-free first-order query bound]
\label{prop:coarse-upper}
Let $d\ge1$, $L,\Delta,\delta>0$, $\bx_0\in\RR^d$, and $L/\sqrt2<\epsilon<L$. Algorithm~\ref{alg:coarse} returns a $(\delta,\epsilon)$-Goldstein stationary point of every $f\in\mathcal F_d(L,\Delta;\bx_0)$ within
\begin{equation}
2+\frac{\Delta L}{\delta(2\epsilon^2-L^2)}
\label{eq:coarse-upper}
\end{equation}
first-order oracle calls.
\end{proposition}

Set $s_{\mathrm{dec}}=(2\epsilon^2-L^2)/L>0$. The algorithm tries a step of length $\delta$ opposite to the returned subgradient and accepts it only when the decrease exceeds $\delta s_{\mathrm{dec}}$.

\begin{algorithm}[H]
\caption{Single-direction descent for coarse accuracy}
\label{alg:coarse}
\begin{algorithmic}[1]
\REQUIRE $L,\delta>0$, $L/\sqrt2<\epsilon<L$, and an initial point $\bx_0$.
\STATE Set $s_{\mathrm{dec}}=(2\epsilon^2-L^2)/L$.
\STATE Query $(f(\bx_0),\bg_0)$; set $\bx=\bx_0$, $a=f(\bx_0)$, and $\bg=\bg_0$.
\LOOP
\IF{$\norm{\bg}\le\epsilon$}
\STATE Return $\bx$.
\ENDIF
\STATE Set $\by=\bx-\delta\bg/\norm{\bg}$. Query $b=f(\by)$ and a subgradient $\bg_{\by}\in\partial f(\by)$.
\IF{$a-b\le\delta s_{\mathrm{dec}}$}
\STATE Return $\bx$.
\ENDIF
\STATE Set $\bx\leftarrow\by$, $a\leftarrow b$, and $\bg\leftarrow\bg_{\by}$.
\ENDLOOP
\end{algorithmic}
\end{algorithm}

\begin{proof}
The stored data satisfy $a=f(\bx)$ and $\bg\in\partial f(\bx)$. If $\norm{\bg}\le\epsilon$, then $\partial f(\bx)\subseteq\partial_\delta f(\bx)$ certifies the return.

Otherwise, let $\bv=\bg/\norm{\bg}$ and $\by=\bx-\delta\bv$. If the trial gives $f(\bx)-f(\by)\le\delta s_{\mathrm{dec}}$, Lebourg's mean value theorem gives $\bz$ in the open segment from $\bx$ to $\by$ and $\mathbf h\in\partial f(\bz)$ such that
\[
f(\by)-f(\bx)=-\delta\ip{\mathbf h}{\bv},
\qquad \ip{\mathbf h}{\bv}\le s_{\mathrm{dec}}.
\]
Both $\bg$ and $\mathbf h$ belong to $\partial_\delta f(\bx)$ and have norm at most $L$. Their average belongs to the same convex set, and
\begin{align*}
\left\|\frac{\bg+\mathbf h}{2}\right\|^2
&=\frac{\norm{\bg}^2+\norm{\mathbf h}^2+2\ip{\bg}{\mathbf h}}4\\
&\le\frac{2L^2+2\norm{\bg}s_{\mathrm{dec}}}{4}\\
&\le\frac{2L^2+2Ls_{\mathrm{dec}}}{4}
=\epsilon^2.
\end{align*}
Here $s_{\mathrm{dec}}>0$ justifies replacing $\norm{\bg}$ by $L$. The return is therefore valid. The vector $\mathbf h$ is supplied by Lebourg's theorem, not by the query at the endpoint. It need not equal $\bg_{\by}$, and the algorithm does not compute it.

Every accepted move decreases the value by more than $\delta s_{\mathrm{dec}}$. Thus any finite number $m$ of accepted moves satisfies
\[
m\delta s_{\mathrm{dec}}\le f(\bx_0)-\inf f\le\Delta,
\qquad
m\le\frac{\Delta}{\delta s_{\mathrm{dec}}}.
\]
Finite work in each iteration and this bound force termination. There is one initial call, one call per accepted move, and at most one final unsuccessful trial, giving at most $m+2$ calls and hence \eqref{eq:coarse-upper}.
\end{proof}

As $\epsilon\downarrow L/\sqrt2$, the denominator in \eqref{eq:coarse-upper} vanishes. The argument gives no guarantee at the endpoint and does not establish optimal behavior near it.

\section{Conclusions and discussion}
\label{sec:scope}

For fixed positive $L,\Delta,\delta$ with $\delta<\Delta/L$ and a fixed ratio $\epsilon/L\in(0,1/3)$, deterministic query complexity is exponential in dimension, even in the local-oracle model of Section~\ref{sec:model}. Function values alone suffice for the same exponential order. In contrast, the first-order method has a dimension-free query bound when $L/\sqrt2<\epsilon<L$.

\subsection{Accuracy and information}
The dependence on accuracy remains unresolved in part. The present bounds do not determine the dimension dependence for $1/3\le\epsilon/L\le1/\sqrt2$, nor do they match the joint dependence on dimension and accuracy. Even at the parameters of Corollary~\ref{cor:theta-d}, the exponential constants are $\tfrac12\log(7/2)$ and $\log33$. The calculation after Theorem~\ref{thm:main} limits the projection estimate for the present construction to $\epsilon<L/3$; it does not establish a transition in the optimization problem. One question is whether the common projection captures the full distance of the Goldstein subdifferential from the origin.

The lower bound is insensitive to the choice between full local information and a value with one Clarke subgradient, because the constructed objective is affine near every queried point in the construction and has a singleton subdifferential there. The zeroth-order upper bound works in both models.

Determinism is essential to the resisting-oracle argument: the prescribed responses fix the history. The construction applied to different random seeds is not required to produce the same objective. Lemma~\ref{lem:direction} therefore does not provide a fixed hard instance with a failure probability for a randomized algorithm.

The bounds concern exact queries, not total computation. No upper bound on direction-net preprocessing time is claimed, and the number of branches in the hard instance depends on the budget. Its interpolation radius can be arbitrarily small when queries are close, although $L_F$ and the common projection do not depend on that radius. No representation-size bound, minimum feature size, or gradient Lipschitz bound is assumed.

\subsection{Quantitative regularity}
Local affine agreement does not make the hard family globally smooth. Appendix~\ref{app:weak-convexity} gives a two-branch member that is not weakly convex for any finite parameter. Results requiring smoothness or weak convexity therefore do not cover the whole family.

Every finite hard instance is nevertheless DC by Appendix~\ref{app:dc}. In comparing with Kong and Lewis \cite{KongLewis}, the relevant issue is quantitative nonconvexity, not DC membership alone. Their directional oracle returns values, directional derivatives, and direction-dependent vectors. The complexity depends on an objective-specific nonconvexity modulus $\Lambda(\delta)$ defined through one-dimensional restrictions. To relate it to a DC decomposition, define
\[
M_{\rm DC}(f)=\inf\{\Lip(q):q\text{ is convex and }f+q\text{ is convex}\}.
\]
Here a feasible $q$ is a convex correction, $\Lip(q)=+\infty$ if it is not globally Lipschitz, and $\inf\varnothing=+\infty$. Proposition~6.3 of Kong and Lewis \cite{KongLewis} gives $\Lambda(\delta)\le\Lip(q)$ for every Lipschitz convex correction, and hence $\Lambda(\delta)\le M_{\rm DC}(f)$. Appendix~\ref{app:dc} gives
\[
M_{\rm DC}(f)\le L(N-1).
\]
This estimate comes from one decomposition. It neither proves that an intrinsic nonconvexity parameter grows with $N$ nor bounds that parameter in terms of $L$ alone. Estimating $M_{\rm DC}(f)$ over all corrections would separate the cost of this representation from the objective's intrinsic nonconvexity. More broadly, it remains to identify quantitative structural assumptions under which a specified oracle admits polynomial dimension dependence.

\Needspace{12\baselineskip}
\appendix

\section{Initial-point conventions and stopping}
\label{app:model}

The following observations relate the prescribed-initial-point model to other initialization and stopping conventions. Choosing the initial point before receiving oracle information incurs no query loss; a free initial response shifts the count by one.

\begin{proposition}[Transfer to an algorithm-chosen initial point]
\label{prop:model-transfer}
Consider a deterministic algorithm that chooses its initial point from $(d,L,\Delta,\delta,\epsilon)$ before seeing any oracle information, makes its first query there, and counts that call. A uniform $T$-call guarantee under that convention induces a uniform $T$-call guarantee in the fixed-initial-point model at the particular initial point it chooses. Consequently, Theorem~\ref{thm:main}, which holds for every prescribed $\bx_0$, transfers to that convention without losing a query.
\end{proposition}

\begin{proof}
Fix the input parameters and the algorithm. If preprocessing does not terminate, there is no finite-return guarantee. Otherwise the chosen initial point $\bx_0^A$ and the internal state before the first query are determined by the input alone. In the prescribed-initial-point model with $\bx_0=\bx_0^A$, perform the same preprocessing and follow the algorithm. The two executions coincide on every $f\in\mathcal F_d(L,\Delta;\bx_0^A)$, with the same charged first call. A uniform $T$-call guarantee would therefore contradict Theorem~\ref{thm:main} at this prescribed point.
\end{proof}

Thus the algorithm-chosen convention in \cite[Section~2]{DNN} incurs no query loss in applying the lower bound.

\paragraph{Free initial responses.}

If the initial response at $\bx_0$ is free and only $J$ subsequent calls are counted, charge once for that response to obtain the convention used here. Conversely, giving the initial response for free removes precisely that call. Hence,
\[
T=J+1,
\qquad
T+1=J+2.
\]
The set $S$ used in the construction has at most $J+2$ points: the initial point, $J$ subsequent locations, and a possible unqueried output. At the fixed accuracy parameters of Corollary~\ref{cor:theta-d}, the lower bound on additional calls is therefore $\lfloor(7/2)^{(d-2)/2}\rfloor-1$, leaving the exponential order unchanged.

\paragraph{Before the initial call.}

Allowing a return before the initial call does not change the lower bound. Such an output $\widehat{\bx}$ is fixed by the input alone. For a budget covered by Theorem~\ref{thm:main}, choose the parameters of Lemma~\ref{lem:parameter-choice} and take
\[
S=\{\bx_0,\widehat{\bx}\}.
\]
This set has at most two distinct points, and hence at most $T+1$. Proposition~\ref{prop:transcript} applies and gives an admissible objective with
\[
\dist(\mathbf0,\partial_\delta f(\widehat{\bx}))\ge\frac{L\alpha}{L_F}>\epsilon.
\]
The no-query execution returns that same invalid point.

If pre-query computation never terminates, it fails on every admissible objective, including the zero function. Otherwise the initial call is reached and the proof of Theorem~\ref{thm:main} applies. Only the stopping convention changes.

\section{Finite deterministic construction of the direction net}
\label{app:net}

The packing argument bounds the number of net points but does not determine when the sphere is covered. A finite decision test and a deterministic selection rule complete the construction. Both are independent of $f$; computational efficiency is not required.

For $d=1$, use the ordered set $(-\mathbf e_1,\mathbf e_1)$ when $0<\eta<2$, and the singleton $\{\mathbf e_1\}$ when $\eta\ge2$. Their cardinalities satisfy the bound in Lemma~\ref{lem:net}. Assume $d\ge2$ below.

\subsection{Rational points on the sphere}
For a rational vector $\mathbf t\in\mathbb Q^{d-1}$, define
\begin{equation*}
R(\mathbf t)
=
\left(
\frac{2\mathbf t}{1+\norm{\mathbf t}^2},
\frac{1-\norm{\mathbf t}^2}{1+\norm{\mathbf t}^2}
\right)
\in\RR^d.
\end{equation*}
With $s=\norm{\mathbf t}^2$, direct substitution gives
\[
\norm{R(\mathbf t)}^2=\frac{4s+(1-s)^2}{(1+s)^2}=1.
\]
Thus $R$ sends rational vectors to rational sphere points. The map extends continuously to $\RR^{d-1}$ and parametrizes the sphere except for $-\mathbf e_d$, with inverse $\mathbf t=\bw'/(1+w_d)$ for $\bw=(\bw',w_d)$ and $w_d\ne-1$. Density of rational parameters shows that $R(\mathbb Q^{d-1})\cup\{-\mathbf e_d\}$ is dense in the sphere. Enumerate $\mathbb Q^{d-1}$ by increasing bounds on absolute numerators and positive denominators, using lexicographic order within each finite bound and skipping repetitions. Applying $R$ and inserting $-\mathbf e_d$ at a fixed position gives a dense enumeration independent of the objective.

\subsection{A stopping test and finite termination}
For an ordered set $\mathcal V=(\bw_1,\ldots,\bw_M)$ of selected rational points, test
\begin{equation}
\exists\bx\in\RR^d:
\quad
\sum_{\ell=1}^d x_\ell^2=1,
\quad
\sum_{\ell=1}^d(x_\ell-w_{j,\ell})^2>\eta^2
\quad(j=1,\ldots,M).
\label{eq:cover-decision}
\end{equation}
Here $x_\ell$ and $w_{j,\ell}$ denote vector coordinates. For $M=0$, only the sphere equation remains, and the test is true. With $\eta$ as a free real parameter, \eqref{eq:cover-decision} is an existential formula with rational polynomial coefficients. Real quantifier elimination \cite{BPR2006} gives an equivalent finite combination of polynomial sign conditions on $\eta$. These can be evaluated by the exact arithmetic and comparisons allowed in Section~\ref{sec:model}, without querying $f$.

If the test is false, $\mathcal V$ is an $\eta$-net. If it is true, the set
\[
U_{\mathcal V}=\{\bx\in\sphere^{d-1}:\norm{\bx-\bw_j}>\eta
\text{ for }j=1,\ldots,M\}
\]
is nonempty and relatively open. It therefore contains a point of the dense enumeration. Append the first enumerated point satisfying the strict inequalities, and repeat the decision test. The search terminates because some admissible point has a finite index. Each selected set is strictly $\eta$-separated and hence has size at most $(1+2/\eta)^d$. Since every decision and selection is finite, the procedure terminates with the required ordered net.

\section{Difference-of-convex structure of the hard family}
\label{app:dc}

A function is \emph{difference-of-convex} (DC) if $f=p-q$ with $p$ and $q$ convex. Equivalently, $q$ is a convex correction making $f+q$ convex; it need not be unique. Each branch of the interpolant is convex, although their minimum need not be. Their sum gives a direct DC decomposition that also accommodates the lower truncation.

\Needspace{20\baselineskip}
\begin{proposition}[The finite hard instance is DC]
\label{prop:dc}
Every finite hard instance produced by Proposition~\ref{prop:transcript} is DC.

If $N=1$, then the construction is
\[
F(\bx)=\mathbf e_1^\top(\bx-\bq_1),
\qquad
f(\bx)
=
\max\left\{
\frac{L}{L_F}\mathbf e_1^\top(\bx-\bq_1),
-\Delta
\right\},
\]
and $f$ is convex. Taking $p=f$ and $q=0$ gives
\[
\Lip(q)=0=L(N-1),
\qquad
\Lip(p)\le L=LN.
\]

If $N\ge2$, let $F_1,\ldots,F_N$ be the branches in \eqref{eq:branch}, let $F=\min_i F_i$, and define
\[
H(\bx)=\sum_{i=1}^N F_i(\bx),
\qquad
K(\bx)=\max_{1\le i\le N}\sum_{j\ne i}F_j(\bx).
\]
Then $H$ and $K$ are convex and $F=H-K$. Moreover, with
\[
p(\bx)
=
\max\left\{
\frac{L}{L_F}H(\bx),
\frac{L}{L_F}K(\bx)-\Delta
\right\},
\qquad
q(\bx)=\frac{L}{L_F}K(\bx),
\]
the functions $p,q$ are convex and satisfy $f=p-q$. The displayed decomposition satisfies
\[
\Lip(q)\le L(N-1),
\qquad
\Lip(p)\le LN.
\]
\end{proposition}

\begin{proof}
The case $N=1$ is given in the statement. For $N\ge2$, every branch is convex: its penalty is the maximum of a convex norm-minus-constant function and zero, and its remaining term is affine. Thus $H$ and $K$ are convex, and
\[
K=\max_i(H-F_i)=H-\min_iF_i=H-F.
\]
Substitution into \eqref{eq:truncation} gives
\[
f=\max\left\{\frac{L}{L_F}(H-K),-\Delta\right\}
=\max\left\{\frac{L}{L_F}H,\frac{L}{L_F}K-\Delta\right\}
-\frac{L}{L_F}K=p-q.
\]
The functions $p$ and $q$ are convex. Since each $F_i$ is $L_F$-Lipschitz, $H$ is $NL_F$-Lipschitz and $K$ is $(N-1)L_F$-Lipschitz. Scaling and taking the finite maximum give the claimed bounds on $\Lip(p)$ and $\Lip(q)$.
\end{proof}

\begin{remark}[Comparison with the Kong--Lewis modulus]
\label{rem:dc-modulus}
Using the notation and DC comparison in Section~\ref{sec:scope}, the correction in Proposition~\ref{prop:dc} gives
\[
\Lambda(\delta)
\le
M_{\rm DC}(f)
\le
\Lip(q)
\le
L(N-1).
\]
This is an upper bound only; it does not establish a lower bound proportional to $N$ for either intrinsic quantity.

For $N\ge2$, the translation identity gives $K(\bx+t\bu)=K(\bx)+(N-1)\alpha t$ for the maximum of the branch sums defining $K$. The displayed correction $q=(L/L_F)K$ therefore satisfies
\[
q(\bx+t\bu)-q(\bx)=\frac{L\alpha(N-1)}{L_F}t.
\]
Because $\norm{\bu}=1$, this gives
\[
\frac{L\alpha(N-1)}{L_F}
\le
\Lip(q)
\le
L(N-1).
\]
For the parameter choice in Lemma~\ref{lem:parameter-choice}, $\alpha/L_F>\epsilon/L$. Thus, for $N\ge2$ and a fixed ratio $\epsilon/L>0$, the Lipschitz constant of this correction is bounded above and below by positive constants times $LN$. These constants can depend on $\epsilon/L$, but not on $d$ or $N$. The lower bound concerns only the displayed correction and does not compare it with all other admissible corrections. Estimating $M_{\rm DC}(f)$ over all corrections, or estimating $\Lambda(\delta)$, requires a separate argument.

\end{remark}

\Needspace{8\baselineskip}

\section{A non-weakly-convex instance}
\label{app:weak-convexity}

A two-branch member of the construction has a downward cusp along a line. Adding a quadratic leaves the jump in the one-sided derivatives unchanged and cannot make the restriction convex.

\begin{proposition}[A two-branch function that is not weakly convex]
\label{prop:not-weakly-convex}
The interpolation formula \eqref{eq:F}, followed by the truncation \eqref{eq:truncation}, can define a function that is not $\varrho$-weakly convex for any finite $\varrho\ge0$.
\end{proposition}

\begin{proof}
Recall that $f$ is $\varrho$-weakly convex if $f(\bx)+(\varrho/2)\norm{\bx}^2$ is convex. Fix $d\ge3$ and $a>0$, and take $\alpha=1/2$, $\gamma=3/4$, and $\beta=2$. Then $c_\gamma=\sqrt7/4$, $\beta c_\gamma>1$, and $L_F=\sqrt{5+2\sqrt3}$. Choose
\[
\bu=\frac12\mathbf e_1+\frac{\sqrt3}{2}\mathbf e_2,
\qquad \bq_+=a\mathbf e_3,\qquad \bq_-=-a\mathbf e_3,\qquad S=\{\bq_+,\bq_-\}.
\]
The two-point set $S$ has only one pairwise difference up to sign, namely $\bq_+-\bq_-=2a\mathbf e_3$. It is perpendicular to $\bu$, and Lemma~\ref{lem:interpolation} applies directly. Its radius is
\[
s_*=2a,\qquad r=\frac{\sqrt7-2}{4}a,\qquad 0<r<a.
\]
Since $\mathbf P\mathbf e_3=\mathbf e_3$ and $\mathbf e_1\perp\mathbf e_3$, the affine terms on the line $\bx=t\mathbf e_3$ satisfy
\[
\mathbf e_1^\top(t\mathbf e_3-\bq_+)=0,\qquad
\mathbf e_1^\top(t\mathbf e_3-\bq_-)=0.
\]
The projected distances are $|t-a|$ and $|t+a|$. For $|t|<a-r$,
\[
F(t\mathbf e_3)=2\min\{|t-a|-r,|t+a|-r\}
=2(a-r)-2|t|>0.
\]
Set $L=\Delta=1$ and $f(\bx)=\max\{F(\bx)/L_F,-1\}$. The truncation is inactive on this interval, giving
\begin{equation}
f(t\mathbf e_3)=\frac{2(a-r)}{L_F}-\frac{2}{L_F}|t|.
\label{eq:downward-cusp}
\end{equation}
For any finite $\varrho\ge0$, \eqref{eq:downward-cusp} shows that the restriction $\psi(t)=f(t\mathbf e_3)+(\varrho/2)t^2$ satisfies
\[
\psi'_-(0)=\frac{2}{L_F}>-\frac{2}{L_F}=\psi'_+(0),
\]
contrary to the ordering of one-sided derivatives of a convex function. Thus no finite weak-convexity parameter is possible.
\end{proof}

This example concerns the interpolation family, not a separate low-dimensional query lower bound. It uses the directional assumptions of Lemma~\ref{lem:interpolation} directly, without invoking the cardinality condition of Lemma~\ref{lem:direction}.

\begingroup
\phantomsection

\endgroup

\end{document}